\documentclass[12pt]{article}

\usepackage{amsmath,amsthm,amssymb,amscd}
\usepackage{easywncy}

\newcommand{\sh}{\mathcyr {sh}}

\newcommand{\I}{{\mathcal I}}

\newcommand{\ZZ}{{\mathcal Z}}
\newcommand{\CC}{{\mathcal C}}
\newcommand{\QQ}{{\mathcal Q}}

\newcommand{\Z}{{\mathbb Z}}

\newcommand{\even}{{\mathbf{e}}}
 \newcommand{\odd}{{\mathbf{o}}}

\newcommand{\C}{{\mathbb C}}
\newcommand{\Q}{{\mathbb Q}}
\newcommand{\R}{\mathbb R}

\newcommand{\DE}{{\mathcal{DE}}_k}

\newcommand{\ve}{\varepsilon}
\newcommand{\SL}{{\rm SL}_2 ({\mathbb Z})}
\newcommand{\PSL}{{\rm PSL}_2 ({\mathbb Z})}
\newcommand{\cz}{\widetilde{\zeta}}

\newcommand{\tr}[1]{{}^t\hspace{0mm}#1}
\newcommand{\PGL}{{\rm PGL}}

\newcommand{\itint}{\mathop {\int \cdots \int}_{1>t_1>t_2>\cdots>t_{r+s}>0} }

\newtheorem*{thm1}{Theorem 1}
\newtheorem*{thm2}{Theorem 2}
\newtheorem*{thm3}{Theorem 3}
\newtheorem*{thm4}{Theorem 4}
\newtheorem*{thm5}{Theorem 5}

\newtheorem*{prop1}{Proposition 1}

\newtheorem*{lem}{Lemma}
\newtheorem*{lem1}{Lemma 1}
\newtheorem*{lem2}{Lemma 2}

\newtheorem*{cor}{Corollary}

\newtheorem*{dfn}{Definition 1}

\title{Double zeta values, double Eisenstein series, and modular forms of level $2$}
\author{Masanobu Kaneko and Koji Tasaka}
\date{}

\begin{document}
\maketitle

\begin{abstract} 
We study the double shuffle relations satisfied by the double zeta values of level 2, and
introduce the double Eisenstein series of level 2 which satisfy the double shuffle relations.
We connect the double Eisenstein series to modular forms of level 2.
\end{abstract}

\section{Introduction}

In \cite{gkz}, H.~Gangl, D.~Zagier and the first author studied in detail the ``double shuffle relations'' satisfied by the  double zeta values
\begin{equation}\label{e_1_1}
\zeta(r,s) = \sum_{m>n>0 } \frac{1}{m^r n^s} \quad (r\ge2, s\ge1), 
\end{equation}
and revealed in particular various connections between the space of double zeta values and the space of modular forms 
as well as their period polynomials on the full modular group $\PSL$.  They also defined  the 
``double Eisenstein series'' and deduced the double shuffle relations for them, and in \cite{kaneko} we illustrated a direct way to connect the double Eisenstein series to the period 
polynomials of modular forms (of level 1).  

In the present paper, we consider the double shuffle relations of level $2$ and study the formal double zeta space, 
whose generators are the formal symbols corresponding to the double zeta values of level 2 (Euler sums) and the
defining relations are the double shuffle relations. One of the relations we obtain in the formal double zeta space
(Theorem~1) has an interesting
application to the problem of representations of integers as sums of squares, and this will be given in the subsequent paper by the second author \cite{t}. We then proceed to
define the double Eisenstein series of level 2 and show that they also satisfy the double shuffle relations (Theorem~3),
and have connections like in the case of level 1 to double zeta values, modular forms, and period polynomials,  of level 2
(Theorem~5 and its corollary).

\section*{Acknowledgments}
This work is partially supported by Japan Society for the Promotion of Science, Grant-in-Aid for Scientific Research (S) 19104002, (B) 23340010ÅD

\section{The double zeta values of level $2$}

The double zeta values of level $2$ we are referring to are the following four types of real numbers given for integers $r\geq2$ and $s\geq1$:
\begin{equation*}
\begin{aligned}
\zeta^{\even\even} (r,s) = \sum_{\begin{subarray}{c} m>n>0 \\ m,\, n:\, {\rm even} \end{subarray}} \frac{1}{m^rn^s}, & \quad 
\zeta^{\even\odd} (r,s) = \sum_{\begin{subarray}{c} m>n>0 \\ m:\, {\rm even},\, n:\, {\rm odd} \end{subarray}} \frac{1}{m^rn^s},  \\
\zeta^{\odd\even} (r,s) = \sum_{\begin{subarray}{c} m>n>0 \\ m:\, {\rm odd},\, n:\, {\rm even} \end{subarray}} \frac{1}{m^rn^s}, &\quad 
\zeta^{\odd\odd} (r,s) = \sum_{\begin{subarray}{c} m>n>0 \\ m,\, n:\, {\rm odd} \end{subarray}} \frac{1}{m^rn^s}.
\end{aligned}
\end{equation*}
These numbers can be written as simple linear combinations of the original multiple zeta values \eqref{e_1_1} and the numbers often referred to as Euler sums defined by
\[ \zeta(r,\overline{s}) = \sum_{m>n>0 } \frac{(-1)^n}{m^r n^s}, \ \zeta(\overline{r},s) = \sum_{m>n>0 } \frac{(-1)^m}{m^r n^s}, \ \zeta(\overline{r},\overline{s}) = \sum_{m>n>0 } \frac{(-1)^{m+n}}{m^r n^s}, \]
and vice versa. We have for instance 
\[ \zeta^{\odd\odd} (r,s) =\frac14\left(\zeta(r,s)-\zeta(\overline{r},s) -\zeta(r,\overline{s})+\zeta(\overline{r},\overline{s}) 
\right) \]
and similarly for other values.
Note that, from the obvious relation 
\[ \zeta(r,s) = \zeta^{\even\even} (r,s) +\zeta^{\even\odd} (r,s) +\zeta^{\odd\even} (r,s)+\zeta^{\odd\odd} (r,s)\]
and 
\[ \zeta(r,s)=2^{r+s}  \zeta^{\even\even} (r,s),\]
we have the relation
\[ (2^{r+s}-1)\zeta^{\even\even} (r,s)=\zeta^{\even\odd} (r,s) +\zeta^{\odd\even} (r,s)+\zeta^{\odd\odd} (r,s).\]
We shall hereafter only consider $\zeta^{\even\odd} (r,s), \zeta^{\odd\even} (r,s)$, and $\zeta^{\odd\odd} (r,s).$
Moreover define 
\[ \zeta^{\even} (k) = \sum_{n>0,\, {\rm even}} \frac{1}{n^k}\quad \text{and}\quad 
\zeta^{\odd} (k) = \sum_{n>0,\, {\rm odd}} \frac{1}{n^k}.\] 
Then in the standard manner we can show the following double shuffle relations.
\begin{prop1}  For positive integers $r, s\ge2$, we have  
\begin{align*}
\zeta^{\odd} (r) \zeta^{\even} (s) &= \zeta^{\odd\even} (r,s) + \zeta^{\even\odd} (s,r)= \sum_{\substack{i+j=r+s\\ i\ge2,\,j\ge1}}\left(\binom{i-1}{r-1} \zeta^{\odd\even} (i,j) + 
\binom{i-1}{s-1} \zeta^{\odd\odd} (i,j)\right), \\
\zeta^{\odd} (r) \zeta^{\odd} (s) &= \zeta^{\odd\odd} (r,s)+\zeta^{\odd\odd} (s,r)+ \zeta^{\odd} (r+s)
=\sum_{\substack{i+j=r+s\\ i\ge2,\,j\ge1}} 
\left( \binom{i-1}{r-1} + \binom{i-1}{s-1} \right) \zeta^{\even\odd} (i,j). \\
\end{align*}
\end{prop1}

\begin{proof} The first equality in each sequence of identities is obtained as usual from the manipulation of the defining series.  For the second, we use the following integral
representations of each zeta value and the shuffle product of integrals:
\begin{align*}
\zeta^{\odd} (k)& = {\mathop {\int \cdots \int}_{1>t_1>t_2>\cdots>t_{k}>0} } \frac{d  t_1}{t_1} \cdot \frac{d  t_2}{t_2} \cdots  \frac{d  t_{k-1}}{t_{k-1}} \cdot \frac{d  t_k}{1-t_k^2},\\
\zeta^{\even} (k) &= {\mathop {\int \cdots \int}_{1>t_1>t_2>\cdots>t_{k}>0} } \frac{d  t_1}{t_1} \cdot \frac{d  t_2}{t_2} \cdots  \frac{d  t_{k-1}}{t_{k-1}} \cdot \frac{t_kd  t_k}{1-t_k^2},\\
\zeta^{\even\odd} (r,s)& = \itint  \frac{d  t_1}{t_1} \cdot \frac{d  t_2}{t_2} \cdots  \frac{d  t_{r-1}}{t_{r-1}} \cdot \frac{d  t_r}{1-t_r^2}\cdot  \frac{d  t_{r+1}}{t_{r+1}} \cdots  \frac{d  t_{r+s-1}}{t_{r+s-1}} \cdot \frac{d  t_{r+s}}{1-t_{r+s}^2} ,\\
\zeta^{\odd\even} (r,s) &=\itint  \frac{d  t_1}{t_1} \cdot \frac{d  t_2}{t_2} \cdots  \frac{d  t_{r-1}}{t_{r-1}} \cdot \frac{d  t_r}{1-t_r^2}\cdot  \frac{d  t_{r+1}}{t_{r+1}} \cdots  \frac{d  t_{r+s-1}}{t_{r+s-1}} \cdot \frac{t_{r+s} d  t_{r+s}}{1-t_{r+s}^2} ,\\
\zeta^{\odd\odd} (r,s)&= \itint  \frac{d  t_1}{t_1} \cdot \frac{d  t_2}{t_2} \cdots  \frac{d  t_{r-1}}{t_{r-1}} \cdot \frac{t_r d  t_r}{1-t_r^2}\cdot  \frac{d  t_{r+1}}{t_{r+1}} \cdots  \frac{d  t_{r+s-1}}{t_{r+s-1}} \cdot \frac{d  t_{r+s}}{1-t_{r+s}^2}.
\end{align*}
The first two are easy to deduce, and to see the rest for double zetas, we use the expressions of those values in terms of the multiple $L$-values studied in \cite{ak2}:
\begin{align*}
\zeta^{\even\odd} (r,s)&=\frac14\left(L_{\sh}(r,s;0,0)-L_{\sh}(r,s;0,1)-L_{\sh}(r,s;1,0)+L_{\sh}(r,s;1,1)\right), \\
\zeta^{\odd\even} (r,s)&=\frac14\left(L_{\sh}(r,s;0,0)+L_{\sh}(r,s;0,1)-L_{\sh}(r,s;1,0)-L_{\sh}(r,s;1,1)\right), \\
\zeta^{\odd\odd} (r,s)&=\frac14\left(L_{\sh}(r,s;0,0)-L_{\sh}(r,s;0,1)+L_{\sh}(r,s;1,0)-L_{\sh}(r,s;1,1)\right).
\end{align*}
Here, the multiple $L$-value (of level 2)
\[   L_{\sh} (k_1,\ldots, k_n;a_1,\ldots,a_n):=\sum_{m_1>\cdots>m_n>0} \frac{(-1)^{(m_1-m_2)a_1 }\cdots (-1)^{(m_{n-1}-m_n)a_2}  (-1)^{m_n a_n}}{ m_1^{k_1} \cdots m_n^{k_n}} \]
is so defined that it has an integral expression similar to that for usual multiple zeta values, and in our case ($n$=2) it is given as
\begin{equation*}
L_{\sh} (r,s;a,b) = \itint \frac{dt_1}{t_1}  \cdots  \frac{dt_{r-1}}{t_{r-1}} \cdot \frac{(-1)^a dt_r}{1-(-1)^a t_r} \cdot \frac{dt_{r+1}}{t_{r+1}}  \cdots  \frac{dt_{r+s-1}}{t_{r+s-1}} \cdot \frac{(-1)^b dt_{r+s}}{1-(-1)^b t_{r+s}},
\end{equation*}
where $a,b=0$ or $1$. The value $L_{\sh} (r,s;0,0)$ is nothing but the double zeta value $\zeta(r,s)$. In terms of the Euler sums mentioned before, we have
\[L_{\sh} (r,s;0,1)= \zeta(r,\overline{s}),\  L_{\sh} (r,s;1,0)=\zeta(\overline{r},\overline{s}),\ L_{\sh} (r,s;1,1)=\zeta(\overline{r},s).\]   Noting the identities
\[ \frac1{1-t^2}=\frac12\left(\frac1{1-t}-\frac{(-1)}{1+t}\right),\quad \frac{t}{1-t^2}=\frac12\left(\frac1{1-t}+\frac{(-1)}{1+t}\right), \]
we obtain the desired integral expressions and hence the proposition by shuffle products.  
\end{proof}

Now we introduce the level 2 version of the formal double zeta space studied in \cite{gkz} as follows.
Let $k>2$ and $\mathcal{DZ}_k$ be the $\Q$-vector space spanned by formal symbols 
$Z^{\even\odd}_{r,s},Z^{\odd\even}_{r,s}$,
$Z^{\odd\odd}_{r,s}$, $P^{\odd\even}_{r,s}$, $P^{\odd\odd}_{r,s}\ ( r,s\ge1,\, r+s=k)$, and $Z^{\odd}_k$ with the 
set of relations
\begin{align}
\label{e_1_7} &\ P^{\odd\even}_{r,s} = Z^{\odd\even}_{r,s} + Z^{\even\odd}_{s,r} = \sum_{\substack{i+j=k\\
i,j\ge1}} \left(\binom{i-1}{r-1} Z^{\odd\even}_{i,j} + \binom{i-1}{s-1} Z^{\odd\odd}_{i,j}\right), \\
\label{e_1_8} &\ P^{\odd\odd}_{r,s} = Z^{\odd\odd}_{r,s}+ Z^{\odd\odd}_{s,r} +Z^{\odd}_{k} = \sum_{\substack{i+j=k\\i,j\ge1}} \left( \binom{i-1}{r-1} +\binom{i-1}{s-1} \right) Z^{\even\odd}_{i,j}
\end{align}
for $r,s\ge1,\, r+s=k$, so that 
\begin{align*}
&\mathcal{DZ}_k = \frac{ \{ \Q\mbox{-linear combinations of }Z^{\even\odd}_{r,s},Z^{\odd\even}_{r,s},
Z^{\odd\odd}_{r,s},P^{\odd\even}_{r,s},P^{\odd\odd}_{r,s}, Z^{\odd}_k  \}}{\langle \mbox{relations\ } \eqref{e_1_7},\eqref{e_1_8} \rangle }.
\end{align*}
Since the elements $P^{\odd\even}_{r,s}$ and $P^{\odd\odd}_{r,s}$ are written in $Z$'s, we can also 
regard the space as given by
\begin{align*}
&\mathcal{DZ}_k = \frac{ \{ \Q\mbox{-linear combinations of }Z^{\even\odd}_{r,s},Z^{\odd\even}_{r,s},
Z^{\odd\odd}_{r,s},Z^{\odd}_k  \}}{\langle \mbox{relations\ } \eqref{ds1},\eqref{ds2} \rangle }
\end{align*}
where the defining relations \eqref{ds1} and \eqref{ds2} are
\begin{align}
\label{ds1} &\ Z^{\odd\even}_{r,s} + Z^{\even\odd}_{s,r} = \sum_{\substack{i+j=k\\
i,j\ge1}} \left(\binom{i-1}{r-1} Z^{\odd\even}_{i,j} + \binom{i-1}{s-1} Z^{\odd\odd}_{i,j}\right), \\
\label{ds2} &\ Z^{\odd\odd}_{r,s}+ Z^{\odd\odd}_{s,r} +Z^{\odd}_{k} = \sum_{\substack{i+j=k\\i,j\ge1}} \left( \binom{i-1}{r-1} +\binom{i-1}{s-1} \right) Z^{\even\odd}_{i,j}.
\end{align}

Note that the relations \eqref{e_1_7} and \eqref{e_1_8} (as well as \eqref{ds1} and \eqref{ds2}) correspond to those in Proposition~1 when $r,s\ge2$,
under the correspondences 
\begin{align*}
&Z^{\even\odd}_{r,s}\longleftrightarrow \zeta^{\even\odd}(r,s),\ \ Z^{\odd\even}_{r,s}
\longleftrightarrow \zeta^{\odd\even}(r,s),\  \ Z^{\odd\odd}_{r,s}\longleftrightarrow \zeta^{\odd\odd}(r,s),
\ \ Z^{\odd}_k\longleftrightarrow \zeta^\odd(k),\\
&\qquad\quad P^{\odd\even}_{r,s}\longleftrightarrow \zeta^\odd(r)\zeta^\even(s),\  P^{\odd\odd}_{r,s}\longleftrightarrow \zeta^\odd(r)\zeta^\odd(s), 
\end{align*}
 because in that case the binomial coefficients for $i=1$ on the right vanishes.  For our later applications
it is convenient to allow the ``divergent'' $Z^{\even\odd}_{1,k-1},  P^{\odd\even}_{1,k-1}$ etc., and in fact 
the double shuffle relations in Proposition~1 can be extended for $r=1$ or $s=1$ by using a suitable regularization procedure
for  $L_{\sh} (1,k-1;0,1)$ etc. developed in \cite{ak2}. Specifically, by setting
\begin{equation}\label{regval}
\zeta^{\odd} (1) := \frac{1}{2} (T - L_{\sh} (1;1) ),\ \zeta^{\even} (1) := \frac{1}{2} (T +L_{\sh} (1;1) )
\end{equation}
$ (L_{\sh} (1;1)=\sum_{m=1}^\infty\frac{(-1)^m}{m}=-\log 2)$ and, for $s\ge2$
\begin{align*}
\zeta^{\even\odd} (1,s)&=\frac12\zeta^\odd(s)T+\frac12L_{\sh} (1;1)\zeta^\odd(s)-\zeta^{\odd\even}(s,1),\\
\zeta^{\odd\even} (1,s)&=\frac12\zeta^\even(s)T-\frac12L_{\sh} (1;1)\zeta^\even(s)-\zeta^{\even\odd}(s,1),\\
\zeta^{\odd\odd} (1,s)&=\frac12\zeta^\odd(s)T-\frac12L_{\sh} (1;1)\zeta^\odd(s)-\zeta^{\odd\odd}(s,1)-\zeta^\odd(s+1)
\end{align*}
where $T$ is a formal variable, the equations in Proposition~1 are valid for all $r,\,s\ge1$ except $(r,s)=(1,1)$.
 
\begin{thm1} Suppose $k$ is even and $k\ge4$.  In $\mathcal{DZ}_k$, we have 

1)
\begin{equation*}
\sum_{\begin{subarray}{c} r=2 \\ r: {\rm even} \end{subarray}}^{k-2} Z^{\odd\odd}_{r,k-r}=\frac{1}{4} Z^{\odd}_k. 
\end{equation*}

2) Each $P^{\odd\even}_{r,k-r}$ with $r$ even can be written as a $\Q$-linear 
combination of $P^{\odd\odd}_{i,j} \ (i,j: {\rm \, even}, i+j=k )$ and $Z^{\odd}_k$
\end{thm1}

\begin{proof} 
Consider the generating functions
\begin{align*}
\ZZ^{\even\odd}_k (X,Y) &= \sum_{r+s=k} Z^{\even\odd}_{r,s} X^{r-1} Y^{s-1} ,  \ \ZZ^{\odd\even}_k (X,Y) = \sum_{r+s=k} Z^{\odd\even}_{r,s} X^{r-1} Y^{s-1}, \\
\ZZ^{\odd\odd}_k (X,Y) &= \sum_{r+s=k} Z^{\odd\odd}_{r,s} X^{r-1} Y^{s-1}. 
\end{align*}
Here and in the following, the sum $\sum_{r+s=k}$ always means $\sum_{r+s=k,\, r,s\ge1}$. The double shuffle relations \eqref{ds1} and \eqref{ds2} are equivalent to the relations
\begin{align}
\label{e_2_1}& \ZZ^{\odd\even}_k(X,Y) +\ZZ^{\even\odd}_k (Y,X) = \ZZ^{\odd\even}_k (X+Y,Y) +\ZZ^{\odd\odd}_k(X+Y,X), \\
\label{e_2_2} & \ZZ^{\odd\odd}_k (X,Y) +\ZZ^{\odd\odd}_k (Y,X) + \ZZ^{\odd}_k \cdot \frac{X^{k-1} - Y^{k-1}}{X-Y} = \ZZ^{\even\odd}_k (X+Y,Y)+\ZZ^{\even\odd}_k (X+Y,X).
\end{align}
Substituting $X=1,Y=0$ in \eqref{e_2_1} and $X=1,Y=-1$ in \eqref{e_2_2},  we respectively obtain
\begin{align}
\label{e_2_7} &Z^{\odd\even}_{k-1,1}+Z^{\even\odd}_{1,k-1} = Z^{\odd\even}_{k-1,1}+\sum_{r=1}^{k-1} Z^{\odd\odd}_{r,k-r},\\
\label{e_2_4} &2 \sum_{r=1}^{k-1} (-1)^{r-1} Z^{\odd\odd}_{r,k-r} +Z^{\odd}_k = 2Z^{\even\odd}_{1,k-1} .
\end{align}
We divide \eqref{e_2_4} by 2 and add \eqref{e_2_7} to obtain 
\[ \frac12 Z^{\odd}_k= 2\sum_{\substack{r=2\\ r:\, even}}^{k-2} Z^{\odd\odd}_{r,k-r}\] 
and hence 1) of Theorem.

To prove 2), we need the following lemma.
\begin{lem1}\label{10} Let $k\geq4$ be an even integer and $a_{i,j},b_{i,j},c_{i,j}$ be rational numbers. Then the following two statements are equivalent.

1) \ The relation
\[ \sum_{i+j=k} a_{i,j} Z^{\even\odd}_{i,j} + \sum_{i+j=k} b_{i,j}Z^{\odd\even}_{i,j} +\sum_{i+j=k} c_{i,j} Z^{\odd\odd}_{i,j} \equiv 0 \pmod{\Q Z^{\odd}_k} \]
holds in $\mathcal{DZ}_k$ (as before $\sum_{i+j=k}$  means $\sum_{i+j=k,\, i,j\ge1}$). 

2) \ There exist some homogeneous polynomials $F,G\in\Q[X,Y]$ of degree $k-2$ such that
\begin{align*}
 &F(Y_1 , X_1)+F(X_2,Y_2)-F(X_2, X_2+Y_2 )  -F(X_3+Y_3,X_3) \\
\notag &+G(X_3,Y_3)+G(Y_3,X_3)- G(X_1,X_1+Y_1) - G(X_1+Y_1,X_1)\\
&=\sum_{i+j=k} \binom{k-2}{i-1} a_{i,j} X_1^{i-1}Y_1^{j-1} + \sum_{i+j=k} \binom{k-2}{i-1} b_{i,j} X_2^{i-1}Y_2^{j-1}+\sum_{i+j=k} \binom{k-2}{i-1} c_{i,j} X_3^{i-1}Y_3^{j-1}.
 \end{align*}
\end{lem1}
\begin{proof}  This is an analogue of Proposition 5.1 in \cite{gkz}. Take $F(X,Y)=\binom{k-2}{r-1} X^{r-1} Y^{s-1}$
(and $G=0$)
and compute the coefficients of $F(Y_1 , X_1)+F(X_2,Y_2)-F(X_2, X_2+Y_2 )  -F(X_3+Y_3,X_3) $ 
using binomial theorem. Then the relation in 1) is exactly (not only $\bmod\  \Q Z_k^\odd$ but as an exact
equality) the relation \eqref{ds1}. 
Similarly, by taking $G(X,Y)=\binom{k-2}{r-1} X^{r-1} Y^{s-1}$
(and $F=0$) and computing the coefficients of $G(X_3,Y_3)+G(Y_3,X_3)- G(X_1,X_1+Y_1) - G(X_1+Y_1,X_1)$,
we see that the relation in 1) is the relation \eqref{ds2} modulo $\Q Z_k^\odd$. Since any relation of the 
form in 1) in $\mathcal{DZ}_k$ should come from a linear combination of \eqref{ds1} and \eqref{ds2} modulo $\Q Z_k^\odd$,
and any homogeneous polynomial is a linear combination of monomials, we obtain the lemma.
\end{proof}

Using the lemma, we are going to produce enough relations of the form
\begin{equation}
\label{eqev} 
\sum_{\begin{subarray}{c} r+s=k \\r,s:\, {\rm even} \end{subarray}} \alpha_{r,s} P^{\odd\even}_{r,s} \equiv  \sum_{\begin{subarray}{c} r+s=k \\ r,s:\, {\rm even}\end{subarray}} \beta_{r,s} P^{\odd\odd}_{r,s} \pmod{\Q Z^{\odd}_k} 
\end{equation}
such that we can solve these in $P^{\odd\even}_{r,s}$.
In view of the relations
\begin{equation} \label{pasz} 
P^{\odd\even}_{r,s} = Z^{\odd\even}_{r,s}+Z^{\even\odd}_{s,r},\  P^{\odd\odd}_{r,s} 
\equiv Z^{\odd\odd}_{r,s} + Z^{\odd\odd}_{s,r}\pmod{\Q Z^{\odd}_k}
\end{equation}
and the lemma, we obtain the relation of the form \eqref{eqev} if we can take $F$ and $G$ in 2) of Lemma 1 
so that the coefficients satisfy
\begin{enumerate}
\item[(i)] $a_{i,j} = b_{j,i}$,
\item[(ii)] $c_{i,j}=c_{j,i}$,
\item[(iii)] $a_{i,j}=b_{i,j}=c_{i,j}=0$ \ for all odd $i, j$.
\end{enumerate}
We now work for convenience with inhomogeneous polynomials. Recall the usual correspondences
$f(x)=F(x,1)$ and $F(X,Y)=Y^{k-2}f(X/Y)$, and the action of the group $\Gamma=\PGL_2(\Z)$
on the space of polynomials of degree at most $k-2$ by
\begin{equation}\label{action} f(x)\Big\vert_{k-2}  \begin{pmatrix} a&b \\c&d \end{pmatrix} = (cx+d)^{k-2} f\left( \frac{ax+b}{cx+d} \right).
\end{equation}
We extend this action to the group ring $\Z[\Gamma]$ by linearity. Set
\[ T=\begin{pmatrix} 1&1\\0&1 \end{pmatrix},\ S=\begin{pmatrix} 0&-1\\1&0 \end{pmatrix},\ \varepsilon=\begin{pmatrix} -1&0\\0&1 \end{pmatrix}, \ \delta=\begin{pmatrix} 0&1\\1&0 \end{pmatrix}.\]
Then the left-hand side of the equation in 2) of Lemma~1 can be written in inhomogeneous form as
\begin{equation}
\label{e_2_15}
\left(f\big| \delta\big.-g \big| \big. (TST+TS\varepsilon )\right)(x_1)+\left(f\big| (1-TST)\big.\right)(x_2) 
-\left(f\big| \big.TS\varepsilon-g\big| \big. (1+\delta)\right)(x_3) .
\end{equation}
(We write  $\big| \big. $ instead of $\big|_{k-2} \big. $.)

\begin{lem2}  Suppose the polynomial $f(x)$ (of degree at most $k-2$) satisfies 
$f\big| \big.TST\varepsilon=f$ and put $g=\frac{1}{2} f\big| \big.  T\varepsilon $. Then the expression
\eqref{e_2_15} gives the coefficients (in Lemma~1-2) ) satisfying the above three conditions (i), (ii), (iii).
\end{lem2}

\begin{proof} Inserting $g=\frac{1}{2} f\big| \big.  T\varepsilon $ into \eqref{e_2_15} and using the assumption
$f\big| \big.TST\varepsilon=f$, which is equivalent to $f\big| \big.TS=f\big| \big.T\varepsilon$ since 
$(T\varepsilon)^2=1$, and also using the identities $TSTST=S, T\varepsilon T=\varepsilon, \varepsilon S=\delta, \delta\varepsilon=\varepsilon\delta=S$ in $\Gamma$, 
we can write \eqref{e_2_15} as
\begin{align}
\label{e_2_16} \left(f\big| \big.  \delta(1-\varepsilon)\right)(x_1) +
\left(f\big| \big.  (1-\varepsilon)\right)(x_2)-\left(f\big| \big.  T(1-\varepsilon)\right)(x_3).
\end{align}
Now the condition (iii) (the polynomial is even) is clear from this (being killed by $1+\varepsilon$),
and the conditions (i) and (ii) are respectively the consequences of the equations
\begin{align*}
f\big| \big.  \delta(1-\varepsilon)\delta&=f\big| \big.  (1-\varepsilon),\\
f\big| \big.  T(1-\varepsilon)\delta&=f\big| \big.  T\delta-f\big| \big.  TS=f\big| \big.  T\varepsilon S
-f\big| \big.  T\varepsilon=f\big| \big. T(1-\varepsilon).
\end{align*}
\end{proof}

Noting $TST\varepsilon= \left(\begin{smallmatrix} -1&0\\-1&1 \end{smallmatrix}\right)$ and hence
\[ (-x+1)\left(\frac{-x}{-x+1}\right)=-x,\ \ \text{and}\ \ (-x+1)\left(\frac{-x}{-x+1}-2\right)=x-2,\]  we see that 
the polynomials $x^r(x-2)^{k-2-r}$ for $r=0,2,\ldots, k-2$ (even) satisfy the condition
$f \big| \big.TST\varepsilon=f$ in Lemma~2. With this choice of $f$ (for  $r=0,2,\ldots, k-4$) 
and $g$ in Lemma~2,
we compute the coefficients in Lemma~1 by noting \eqref{pasz}, \eqref{e_2_16} and by 
using
 \begin{align*} &x^r(x-2)^{k-2-r}\vert(1-\ve)= x^r(x-2)^{k-2-r}-x^r(x+2)^{k-2-r}\\
 &=-\sum_{\substack{i=1\\i: odd}}^{k-2-r-1} \binom{k-2-r}{i} 2^{k-1-r-i}x^{r+i} \\
 &=- \sum_{\substack{i=r+2\\i: even}}^{k-2} \binom{k-2-r}{i-1-r} 2^{k-i} x^{i-1} \quad\  (r+i \rightarrow i-1) \\
 &= -  \binom{k-2}{r}^{-1}\sum_{\substack{i=r+2\\ i:even}}^{k-2} \binom{k-2}{i-1}  \binom{i-1}{r} 2^{k-i}x^{i-1},
 \end{align*}
to obtain a relation of the form
\[  \sum_{\substack{i=r+2\\i: even}}^{k-2} \binom{i-1}{r} 2^{k-i}P_{i,k-i}^{\odd\even}
\equiv \text{linear combination of } P_{\text{even}, \text{even}}^{\odd\odd}\pmod {\Q Z_k^\odd}. \] 
When we put  $r=k-4,\ldots,2,0$, we can solve these congruences successively in each $P_{i,k-i}^{\odd\even}$
for $i=k-2,k-4,\ldots, 2$. This completes the proof of Theorem~1.

\end{proof}

\section{The  double Eisenstein series of level $2$}
\subsection{Definition and the double shuffle relations}

We introduce the double Eisenstein series of level 2 and first show that they satisfy the double shuffle 
relations.

Let $\mathbf{ev}$ (resp. $\mathbf{od}$) be the set of even (resp. odd) integers and $\tau$ a variable in the upper half-plane. Define the three
double Eisenstein series $G^{\even\odd}_{r,s} (\tau), G^{\odd\even}_{r,s} (\tau)$, and $G^{\odd\odd}_{r,s} (\tau)$ by 
\begin{equation}\label{e_1_10} 
\begin{aligned}
G^{\even\odd}_{r,s} (\tau)& := (2\pi i)^{-r-s}\!\!\!\sum_{\begin{subarray}{c} \lambda > \mu>0 \\ \lambda \in \mathbf{ev}\cdot\tau+\mathbf{ev} \\ \mu \in \mathbf{ev}\cdot\tau + \mathbf{od} \end{subarray}} 
\frac{1}{\lambda^r \mu^s} =(2\pi i)^{-r-s}\!\!\!\!\!\sum_{\begin{subarray}{c} m\tau+n > m'\tau+n' >0 \\ m\in \mathbf{ev},n \in \mathbf{ev}\\ m' \in \mathbf{ev}, n' \in \mathbf{od} \end{subarray}} 
\frac{1}{(m\tau+n)^r (m'\tau+n')^s},\\
G^{\odd\even}_{r,s} (\tau)& := (2\pi i)^{-r-s}\!\!\!\sum_{\begin{subarray}{c} \lambda > \mu>0 \\ \lambda \in \mathbf{ev}\cdot\tau+\mathbf{od} \\ \mu \in \mathbf{ev}\cdot\tau + \mathbf{ev} \end{subarray}} 
\frac{1}{\lambda^r \mu^s},\quad 
G^{\odd\odd}_{r,s} (\tau) := (2\pi i)^{-r-s}\!\!\!\sum_{\begin{subarray}{c} \lambda > \mu>0 \\ \lambda \in \mathbf{ev}\cdot\tau+\mathbf{od} \\ \mu \in \mathbf{ev}\cdot\tau + \mathbf{od} \end{subarray}} 
\frac{1}{\lambda^r \mu^s}.
\end{aligned}
\end{equation}
Here, the positivity $m\tau+n>0$ of a lattice point means either $m>0$ or $m=0, n>0$, and $m\tau+n>m'\tau+n'$ means $(m-m')\tau+(n-n')>0$. We assume  $r\ge3$ and $s\ge2$ for the absolute convergence.

All the series in  \eqref{e_1_10} is easily seen to be invariant the translation $\tau\to\tau+1$, and hence have
Fourier expansions. The Fourier series developments can be deduced in a quite similar manner to the full modular case \cite{gkz}. 

\begin{thm2}\label{6}
Let $r\geq 3$ and $s\geq2$ be integers and set $k=r+s$.  We have the following $q$-series expansions ($q=e^{2\pi i \tau}$).

\begin{align*}
G^{\even\odd}_{r,s}(\tau)  =& \cz^{\even\odd} (r,s) + g_{r,s}^{\even\odd} (q) \\
& +\sum_{ \begin{subarray}{c} p+h=k \\ p>1 \end{subarray} } \left\{ \left( (-1)^s \binom{p-1}{s-1} +\delta_{p,s} \right) \cz^{\odd} (p) g_h^{\even} (q) +(-1)^{p+r} \binom{p-1}{r-1}\cz^{\odd} (p) g_h^{\odd}(q) \right\}, \\
G_{r,s}^{\odd\even} (\tau)=&  \cz^{\odd\even} (r,s) + g_{r,s}^{\odd\even} (q) \\
& +\sum_{ \begin{subarray}{c} p+h=k \\ p>1 \end{subarray} } \left\{  (-1)^s \binom{p-1}{s-1} \cz^{\odd} (p) g_h^{\odd}(q)  +\delta_{p,s}\cz^{\even} (p) g_h^{\odd} (q) +(-1)^{p+r} \binom{p-1}{r-1} \cz^{\odd} (p) g_h^{\even}(q) \right\}, \\
G_{r,s}^{\odd\odd} (\tau) =&  \cz^{\odd\odd} (r,s) + g_{r,s}^{\odd\odd} (q) \\
&+\sum_{ \begin{subarray}{c} p+h=k \\ p>1 \end{subarray} } \left\{ \left((-1)^s \binom{p-1}{s-1} +(-1)^{p+r} \binom{p-1}{r-1}\right) \cz^{\even} (p) g_h^{\odd} (q) +\delta_{p,s}\cz^{\odd} (p) g_h^{\odd}(q) \right\}, 
\end{align*}
where $\delta_{p,s}$ is Kronecker's delta, $\cz^{\ast\ast} (r,s)=(2\pi i)^{-r-s} \zeta^{\ast\ast} (r,s)$ and $\cz^{\ast} (k) = (2\pi i)^{-k} \zeta^{\ast} (k)$ $(\ast=\even\text{ or }\odd)$, and the $g$'s are the 
following $q$-series:
\begin{align*}
g_{r,s}^{\even\odd} (q) &= - \frac{(-1)^{r+s}}{2^{r+s}(r-1)!(s-1)!} \sum_{\begin{subarray}{c} m>m'>0 \\ u,v>0 \end{subarray} } u^{r-1} (-v)^{s-1} q^{um+vm'},  \\
g_{r,s}^{\odd\even} (q) &= -\frac{(-1)^{r+s}}{2^{r+s}(r-1)!(s-1)!} \sum_{\begin{subarray}{c} m>m'>0 \\ u,v>0 \end{subarray} }(-u)^{r-1} v^{s-1}q^{um+vm'},  \\
g_{r,s}^{\odd\odd} (q) &= \frac{(-1)^{r+s}}{2^{r+s}(r-1)!(s-1)!} \sum_{\begin{subarray}{c} m>m'>0 \\ u,v>0 \end{subarray} } (-u)^{r-1} (-v)^{s-1} q^{um+vm'}, 
\end{align*}
and 
\[g_r^{\even} (q) = \frac{(-1)^r}{2^r(r-1)!}\sum_{u,m>0} u^{r-1} q^{um},\ g_r^{\odd} (q) =\frac{(-1)^r}{2^r(r-1)!} \sum_{u,m>0} (-1)^u u^{r-1}  q^{um}. \]
\end{thm2}

\begin{proof}
Put, for positive integers $r$ and $s$,
\[
\varphi_r^{\even} (q) = \frac{(-1)^r }{2^r (r-1)!} \sum_{u>0} u^{r-1} q^{u/2}, \ \varphi_r^{\odd} (q) = \frac{(-1)^{r} }{2^r (r-1)!}  \sum_{u>0} (-1)^u u^{r-1} q^{u/2}.
\]
Then the series $g_r^\ast(q), g_{r,s}^{\ast\ast}(q)$ in the theorem can be written by using $\varphi_r^{\ast}(q) $ as
\begin{align*} 
g_r^{\even} (q)& = \sum_{m>0} \varphi_r^{\even} (q^{2m}),\ g_r^{\odd} (q) = \sum_{m>0} \varphi_r^{\odd} (q^{2m}), \\ 
g_{r,s}^{\even\odd} (q) &= \sum_{m>m'>0 }\varphi_r^{\even} (q^{2m}) \varphi_s^{\odd} (q^{2m'}) ,\ g_{r,s}^{\odd\even} (q) = \sum_{m>m'>0 }\varphi_r^{\odd} (q^{2m}) \varphi_s^{\even} (q^{2m'}), \\
g_{r,s}^{\odd\odd} (q) &= \sum_{m>m'>0 }\varphi_r^{\odd} (q^{2m}) \varphi_s^{\odd} (q^{2m'}), \ g_{r,s}^{\even\even} (q) = \sum_{m>m'>0 }\varphi_r^{\even} (q^{2m}) \varphi_s^{\even} (q^{2m'}).
\end{align*}
The computation of the Fourier series can be carried out in a completely similar fashion as done in \cite{gkz}, dividing the sum of the defining series into four terms, according as $m=m'=0, m=m'>0, m>m'=0,
m>m'>0$.
For instance, in the case of $G^{\even\odd}_{r,s}(\tau)$, we compute
\[ G^{\even\odd}_{r,s}(\tau) =\Biggl\{ \sum_{\begin{subarray}{c}  m=m'=0 \\ n>n' >0 \\ m,m',n\in \mathbf{ev} \\ n' \in \mathbf{od} \end{subarray}} + \sum_{\begin{subarray}{c}  m=m'>0 \\ n>n' \\ m,m',n\in \mathbf{ev} \\ n' \in \mathbf{od} \end{subarray}} + \sum_{\begin{subarray}{c}  m>m'=0 \\  n' >0  \\ m,m',n\in \mathbf{ev} \\ n' \in \mathbf{od} \end{subarray}} +\sum_{\begin{subarray}{c}  m>m'>0  \\ m,m',n\in \mathbf{ev} \\ n' \in \mathbf{od}  \end{subarray}} \Biggr\} \frac{(2\pi i)^{-r-s}}{(m\tau+n)^r (m'\tau+n')^s},\]
using the partial fraction decomposition 
\begin{align*}
\notag \frac{1}{(\tau+n)^r (\tau + n')^s} =& (-1)^s \sum_{i=0}^{r-1} \binom{s+i-1}{i} \frac{1}{(\tau+n)^{r-i}}\cdot \frac{1}{(n-n')^{s+i}} \\
 &+ \sum_{j=0}^{s-1} (-1)^j \binom{r+j-1}{j} \frac{1}{(\tau+n')^{s-j}}\cdot \frac{1}{(n-n')^{r+j}} 
\end{align*}
and the formulas
\begin{align*} 
\sum_{n\in \Z } \frac{1}{(\tau+2n)^r} &= \frac{(-2\pi i)^r}{2^r(r-1)!} \sum_{u>0} u^{r-1} q^{u/2} = (2\pi i)^r \varphi_r^{\even} (q) \ \ \ (r\geq2),\\
\sum_{n\in \Z } \frac{1}{(\tau+2n+1)^r} &= \frac{(-2\pi i)^r}{2^r(r-1)!} \sum_{u>0} (-1)^u u^{r-1} q^{u/2} = (2\pi i)^r \varphi_r^{\odd} (q) \ \ \ (r\geq2)
\end{align*}
(consequences of the standard Lipschitz formula, and when $r=1$ we use 
\begin{align*}
\lim_{N\rightarrow \infty} \sum_{n=-N}^N \frac{1}{\tau+2n} & = -\frac{\pi i}{2} +\frac{(-2\pi i )}{2} \sum_{u>0} q^{u/2} = -\frac{\pi i}{2} + (2\pi i) \varphi_{1}^{\even} (q),\\
\lim_{N\rightarrow \infty} \sum_{n=-N}^N \frac{1}{\tau+2n+1}& = -\frac{\pi i}{2} +(2\pi i) \varphi_1^{\odd} (q) 
\end{align*}
instead).  We leave the details to the reader.
\end{proof}

We remark that each series in Theorem~2 is in $\R+q\Q[[q]]+\sqrt{-1}\R[[q]]$, and the terms in $\sqrt{-1}\R[[q]]$ (``imaginary part'')
only come from the terms having $\cz^{\ast}(p)$ with odd $p$ as coefficients. 

Now we extend the definition of the double Eisenstein series for any (non-converging) $r,s\ge1$ (except $r=s=1$), by using $q$-series.
For this, we separately define the imaginary part and the ``combinatorial part''. 
First we define the imaginary parts as
\begin{align*}
I^{\even\odd}_{r,s}(q)&= \sum_{\begin{subarray}{c} p+h=k \\ p :odd \end{subarray}} \left\{ \left( (-1)^s \binom{p-1}{s-1} +\delta_{p,s} \right) \cz^{\odd} (p) g_h^{\even} (q) +(-1)^{p+r} \binom{p-1}{r-1}\cz^{\odd} (p) g_h^{\odd}(q) \right\}, \\
I^{\odd\even}_{r,s}(q)&= \sum_{\begin{subarray}{c} p+h=k \\ p :odd \end{subarray}} \left\{  (-1)^s \binom{p-1}{s-1} \cz^{\odd} (p) g_h^{\odd}(q)  +\delta_{p,s}\cz^{\even} (p) g_h^{\odd} (q) +(-1)^{p+r} \binom{p-1}{r-1} \cz^{\odd} (p) g_h^{\even}(q) \right\}, \\
I^{\odd\odd}_{r,s}(q)  &= \sum_{\begin{subarray}{c} p+h=k \\ p :odd \end{subarray}} \left\{ \left( (-1)^s \binom{p-1}{s-1} +(-1)^{p+r} \binom{p-1}{r-1}  \right) \cz^{\even} (p) g_h^{\odd} (q) +\delta_{p,s}\cz^{\odd} (p) g_h^{\odd}(q) \right\}.
\end{align*}
The sum is over $p,h\ge1$ with $p$ odd.  Note that the regularized values $\cz^{\odd} (1)$ and $\cz^{\even} (1)$ are defined by \eqref{regval} and thus for any positive integers $r,s$, these series are elements of $\sqrt{-1}\R[T][[q]]$. Secondly, we define the part in $q\Q[[q]]$ which is referred to as the combinatorial double Eisenstein series. Put
\begin{align*}
\beta_{r,s}^{\even\odd} (q) &= \sum_{p+h=r+s} \left\{ \left( (-1)^s \binom{p-1}{s-1} +\delta_{p,s} \right) \beta_p^{\odd} g_h^{\even} (q) +(-1)^{p+r} \binom{p-1}{r-1} \beta_p^{\odd} g_h^{\odd}(q) \right\}, \\
\beta_{r,s}^{\odd\even} (q) &= \sum_{p+h=r+s} \left\{  (-1)^s \binom{p-1}{s-1} \beta_p^{\odd} g_h^{\odd}(q)  +\delta_{p,s} \beta_p^{\even} g_h^{\odd} (q) +(-1)^{p+r} \binom{p-1}{r-1} \beta_p^{\odd} g_h^{\even}(q) \right\}, \\
\beta_{r,s}^{\odd\odd} (q) &= \sum_{p+h=r+s} \left\{ \left( (-1)^s \binom{p-1}{s-1} +(-1)^{p+r} \binom{p-1}{r-1}  \right) \beta_p^{\even} g_h^{\odd} (q) +\delta_{p,s} \beta_p^{\odd} g_h^{\odd}(q) \right\}, 
\end{align*}
where 
\[\beta_r^{\even} = -\frac{B_r}{2^{r+1}\cdot r!}, \ \beta_r^{\odd} = - \frac{(1-2^{-r}) B_r}{2\cdot r!}\quad (B_r= \text{the Bernoulli number}), \]
and as before the condition ``$p+h=r+s$'' includes ``$p,h\geq1$''.
Let \[ \overline{g}_r^{\even} (q) := -\sum_{m>0} m \varphi_{r+1}^{\even} (q^{2m}), \ \overline{g}_r^{\odd} (q) := - \sum_{m>0} m \varphi_{r+1}^{\odd} (q^{2m})\quad (r\ge0), \]
and for integers $r,s\geq1$ let  
\begin{align*}
\varepsilon_{r,s}^{\even\odd} (q) =& \delta_{r,2} \overline{g}_s^{\odd} (q) -\delta_{r,1}\overline{g}_{s-1}^{\odd} (q) + \delta_{s,1} (\overline{g}_{r-1}^{\even} (q) +g_r^{\even} (q))+\delta_{r,1}\delta_{s,1} \alpha_1,\\
\varepsilon_{r,s}^{\odd\even} (q) =& \delta_{r,2} \overline{g}_s^{\even} (q) -\delta_{r,1}\overline{g}_{s-1}^{\even} (q) + \delta_{s,1} (\overline{g}_{r-1}^{\odd} (q) +g_r^{\odd} (q))+\delta_{r,1}\delta_{s,1} \alpha_2,\\
\varepsilon_{r,s}^{\odd\odd} (q) =& \delta_{r,2} \overline{g}_s^{\odd} (q) -\delta_{r,1}\overline{g}_{s-1}^{\odd} (q) + \delta_{s,1} (\overline{g}_{r-1}^{\odd} (q) +g_r^{\odd} (q))+\delta_{r,1}\delta_{s,1} \alpha_3,
\end{align*}
where 
\begin{equation}
\label{e_3_8} \alpha_1=\overline{g}_0^{\odd} (q)- \frac{1}{2} \overline{g}_0^{\even}(q), \alpha_2=-\alpha_1, \ \alpha_3 = 4g_2^{\odd} (q) +\dfrac{1}{2} \overline{g}_0^{\even} (q) .
\end{equation}
Note that each $\varepsilon^{\ast,\ast}_{r,s}$ is $0$ when $r\geq3$ and $s\geq2$. The combinatorial double Eisenstein series are then defined, for positive integers $r,s\geq1$, by
\begin{align*}
C^{\even\odd}_{r,s}(q)  =& g_{r,s}^{\even\odd} (q) +\beta_{r,s}^{\even\odd} (q) +\dfrac{1}{4} \varepsilon_{r,s}^{\even\odd} (q), \\
C^{\even\odd}_{r,s}(q) =& g_{r,s}^{\odd\even} (q) +\beta_{r,s}^{\odd\even} (q) +\dfrac{1}{4} \varepsilon_{r,s}^{\odd\even} (q), \\
C^{\odd\odd}_{r,s}(q) =&  g_{r,s}^{\odd\odd} (q) +\beta_{r,s}^{\odd\odd} (q) +\dfrac{1}{4} \varepsilon_{r,s}^{\odd\odd} (q).
\end{align*}
Lastly, the constant term of the double Eisenstein series is given by the (regularized) double zeta values. 
\begin{dfn}\label{15} For any integers $r,s\geq1$ with $(r,s)\neq(1,1)$, we define 
\begin{align*}
G^{\even\odd}_{r,s}(q)&:=\cz^{\even\odd} (r,s) +C^{\even\odd}_{r,s}(q) +I^{\even\odd}_{r,s}(q) ,\\
G^{\odd\even}_{r,s}(q)&:=\cz^{\odd\even} (r,s) +C^{\odd\even}_{r,s}(q)+I^{\odd\even}_{r,s}(q) ,\\
G^{\odd\odd}_{r,s}(q)&:=\cz^{\odd\odd} (r,s) +C^{\odd\odd}_{r,s}(q) +I^{\odd\odd}_{r,s}(q).
\end{align*}
\end{dfn}

To state the double shuffle relations in the forms \eqref{e_1_7} and \eqref{e_1_8} for these series, we need usual Eisenstein series for the congruence
subgroup $\Gamma_0(2)=\left\{\left(\begin{smallmatrix} a & b \\ c & d \end{smallmatrix}\right)\in \SL\,\vert\, c\equiv0\bmod 2\right\}$.
For each integer $k\geq3$, let the series $G_k^{(i\infty)} (\tau)$ and $G_k^{(0)} (\tau)$ be defined by
\[ 
G_k^{(i\infty)} (\tau) := \sum_{\begin{subarray}{c} \lambda >0 \\ \lambda \in \mathbf{ev}\cdot \tau + \mathbf{od} \end{subarray}} \frac{1}{\lambda^k}
=\sum_{\begin{subarray}{c} m\tau+n >0 \\ m:\,even,\,n:\,odd \end{subarray}} \frac{1}{(m\tau+n)^k},\] 
and
\[
G_k^{(0)} (\tau) :=  \sum_{\begin{subarray}{c} \lambda >0 \\ \lambda \in \mathbf{od}\cdot \tau + \Z \end{subarray}} \frac{1}{\lambda^k}=
\sum_{\begin{subarray}{c} m\tau+n >0 \\ m:\,odd \end{subarray}} \frac{1}{(m\tau+n)^k}.
\]
When $k\geq4$ is even, the functions $G_k^{(i\infty)} (\tau)$ and $G_k^{(0)} (\tau)$ are the Eisenstein series for $\Gamma_0(2)$ associated to cusps $i\infty$ and $0$ respectively, 
and as such they are modular of weight $k$ with respect to $\Gamma_0(2)$.
The Fourier series of $G_k^{(i\infty)} (\tau)$ and $G_k^{(0)} (\tau)$ are given as follows.  Let  $G_k(\tau)$ be the Eisenstein series of weight $k$ for $\SL$:  
\begin{equation}\label{eisen}
G_k(\tau): =\sum_{\Z\tau+\Z \ni m\tau+n>0 } \frac1{(m\tau+n)^{k}}= \zeta(k)+\frac{(-2\pi i)^k}{(k-1)!} \sum_{n\geq1} \sigma_{k-1} (n) q^n,\quad (\sigma_{k-1} (n)= \sum_{d|n} d^{k-1}).
\end{equation}
(Note that this gives a non-zero function even when $k$ is odd.)  With this we have
\begin{equation}\label{e_3_2}
\begin{aligned}G_k^{(i\infty)} (\tau) &= G_k(2\tau)-2^{-k} G_k(\tau) = \zeta^{\odd} (k) + \frac{(-2\pi i)^k}{2^k(k-1)!}
\sum_{n\ge1}\Bigl(\sum_{d|n}(-1)^dd^{k-1}\Bigr)q^n ,\\
G_k^{(0)} (\tau) &= G_k(\tau)- G_k(2\tau) = \frac{(-2\pi i)^k}{(k-1)!} \sum_{n\geq1}
\Bigl(\sum_{\substack{d|n\\ n/d: odd}}d^{k-1}\Bigr)q^n. 
\end{aligned}
\end{equation}
We define the $q$-series $G_k(q), G_k^{(i\infty)} (q)$, and $G_k^{(0)} (q)$ for any $k\ge1$ by the (convergent) $q$-series
on the right-hand sides of \eqref{eisen} and \eqref{e_3_2}, with the regularized values  \eqref{regval} and $\zeta(1)=T (=\zeta^\even(1)+\zeta^\odd(1))$.
Finally we set 
\begin{align*} 
G_k^\odd (q) &=(2\pi i)^{-k} G_k^{(i\infty)}(q)=\cz^{\odd} (k)+g_k^{\odd} (q),\\  
G_k^\even(q) &=2^{-k}(2\pi i)^{-k}  G_k (q)=\cz^\even(k)+g_k^\even(q) .
\end{align*}

\begin{thm3}\label{3_3} For any integers $r,s\geq1$ with $( r,s)\neq (1,1)$, we have 
\begin{align*}
G_r^\odd(q)G_s^\even(q) + \dfrac{1}{4} (\delta_{r,2} \overline{g}_s^{\even} (q) + \delta_{s,2} \overline{g}_r^{\odd} (q) ) &= G^{\odd\even}_{r,s}(q) +G^{\even\odd}_{s,r}(q)\\ 
&= \sum_{i+j=r+s} \left(\binom{i-1}{r-1} 
G^{\odd\even}_{i,j}(q) +\binom{i-1}{s-1} G^{\odd\odd}_{i,j}(q)\right), \\
G^\odd_r(q)G^\odd_s(q) + \frac{1}{4} (\delta_{r,2} \overline{g}_s^{\odd} (q) + \delta_{s,2} \overline{g}_r^{\odd} (q) )
&= G^{\odd\odd}_{r,s}(q)+G^{\odd\odd}_{s,r}(q)+G^\odd_{r+s}(q) \\
&= \sum_{i+j=r+s} \left( \binom{i-1}{r-1}+\binom{i-1}{s-1} \right) G^{\even\odd}_{i,j}(q).
\end{align*}
\end{thm3}
The proof of the theorem will be postponed to \S\ref{proof}.

\subsection{Double Eisenstein series and period polynomials}

In this subsection, we describe a mysterious connection between our double Eisenstein series and the period polynomials associated to cusp forms on $\Gamma_0(2)$. 
This kind of connection was first observed in the full modular case \cite{kaneko}, which will be recalled briefly in the appendix 
\S\ref{appendix} for the convenience of the reader because
the reference \cite{kaneko} circulated only among participants of the conference.

Let us recall the theory of period polynomials for $\Gamma_0(2)$ given in \cite{ki} and \cite{fy}.  We follow the formulation of \cite{KoZa}.
Recall the group $\Gamma_0(2)$ is generated by two elements (see e.g. \cite[Theorem 4.3]{ap}))
\[    T= \begin{pmatrix} 1&1\\ 0&1 \end{pmatrix},\ M= \begin{pmatrix} -1&-1\\ 2&1 \end{pmatrix}. \]
Let $k$ be a positive even integer and $V_k$ be the space of polynomials with rational coefficients of degree at most $k-2$:
\[ V_k:=\left\{ P(X)\in\Q[X]\,\vert\, \deg (f)\le k-2\right\}. \] 
The group $\Gamma_0(2)$ acts on $V_k$ as in (\ref{action}) and this action extends to that of the group ring $\Z[\Gamma_0(2)]$ as usual.  Consider the subspace 
$W_k$ of $V_k$ defined by
\[ W_k:=\left\{ P\in V_k\,\big\vert\, P\vert_{k-2} (1-T)(1+M)=0 \right\}. \] 
For a cusp form $f\in S_k(2):=\left\{\text{\,the space of cusp forms on }\Gamma_0(2)\,\right\}$, the period polynomial $r_f(X)$ is
given by 
\[ r_f(X):=\int_0^{i\infty} f( \tau) (X-\tau)^n d\tau. \]
It is implicitly shown in the proof of Proposition 3 in \cite{ki} that 
\[ r_f(X)\in W_k\otimes \C.  \]
Now we consider the even and odd parts of polynomials separately. Put $\varepsilon=\left(\begin{smallmatrix} -1&0\\0&1 \end{smallmatrix}\right)$. By the identity
\[ \ve (1-T)(1+M)=-(1-T)(1+M)T^{-1}\ve \]
(every matrix identity is regarded projectively, i.e., as that in $\Gamma_0(2)/\left\{\pm1\right\}$), we see that if $P\in W_k$ then
$P\vert (1\pm\ve)\in W_k$, and so we have the direct sum decomposition 
\[ W_k=W_k^+\oplus W_k^-, \]
where $W_k^+$ (resp. $W_k^-$) is the even (resp. odd) part of $W_k$:
\[ W_k^{\pm}:=\left\{ P\in V_k\,\big\vert\  P\vert\ve=\pm P \text{ and } P\vert (1-T)(1+M)=0\,\right\}. \]
We also denote by $r_f^{\pm}(X)$ the even and odd part of $r_f(X)$,
\[  r_f^{\pm}(X):=\frac12r_f(X)\vert (1\pm\ve), \]
and by $r^\pm$ the map
\[ r^\pm : S_k(2)\ni f\longmapsto r_f^\pm(X)\in W_k^\pm\otimes\C. \]
For the space $W_k^+$ of even polynomials, we have two obvious elements $1$ and $X^{k-2}$. This is clear for $1$ because $1\vert (1-T)=0$.   For $X^{k-2}$, we note the
identity
\[ (1-T)(1+M)=(1-TM)(1+M)\quad (\text{because }M^2=1) \]
and $TM=\left(\begin{smallmatrix} 1&0\\2&1 \end{smallmatrix}\right)$ and thus $X^{k-2}\vert TM=X^{k-2}$. Hence, we have the decomposition
\[ W_k^+=\Q\cdot1\oplus\Q\cdot X^{k-2}\oplus W_k^{+,0}, \]
where
\[ W_k^{+,0}:=\Biggl\{ P(X)\in V_k\,\Big\vert\, P(X)=\sum_{\substack{i=2\\ even}}^{k-4} a_i X^i,\ P\vert(1-T)(1+M)=0\Biggr\}. \]
Let $r^{+,0}$ be the map $S_k(2)\to W_k^{+,0}$ obtained by the composition of $r^+$ and the natural projection $W_k^+\to W_k^{+,0}$.
From the works of Imamo\={g}lu-Kohnen \cite{ki} and Fukuhara-Yang \cite{fy}, we obtain the following

\begin{thm4}  For even $k$, the two maps
\[ r^+: S_k(2) \longrightarrow W_k^{+,0}\otimes\C\quad\text{and}\quad r^-: S_k(2)\longrightarrow W_k^-\otimes\C \]
are isomorphisms of vector spaces.
\end{thm4}

\begin{proof}  We know from \cite{ki} and \cite{fy} that both maps are injective. So all we have to show is
the dimensions of the target spaces are equal to the dimension of $S_k(2)$, which is equal to $[k/4]-1$.
We only calculate the dimension of $W_k^{+,0}$, since the other is similar and only the former is relevant to 
the subsequent story involving the double Eisenstein series.

Put $T'=\left(\begin{smallmatrix} 1&0\\-1&1 \end{smallmatrix}\right)$. Obviously $P=0$ is equivalent to $P\vert T'=0$.
For an even polynomial \[P(X)=\sum_{\substack{2\le i\le k-4\\ i: even}} a_i X^i\] with no constant term and no $X^{k-2}$ term,
we compute the condition $P\vert (1-T)(1+M)T'=0$ for $P$ being in $W_k^{+,0}$.  By
\[ (1-T)(1+M)T'=T'-TT'-{}^tT+MT'\quad (TMT'={}^tT=\begin{pmatrix} 1&0\\ 1&1 \end{pmatrix}) \] and
\[ TT'=\begin{pmatrix} 0&1\\ -1&1 \end{pmatrix},\ MT'=\begin{pmatrix} 0&-1\\ 1&1 \end{pmatrix}, \]
the condition becomes
\[ (-X+1)^{k-2}\left(P\bigl(\frac{X}{-X+1}\bigr)-P\bigl(\frac{1}{-X+1}\bigr)\right)
-(X+1)^{k-2}\left(P\bigl(\frac{X}{X+1}\bigr)-P\bigl(\frac{-1}{X+1}\bigr)\right)=0, \]
which is written as
\[ \sum_{\substack{2\le i\le k-4\\ i: even}} a_i(X^i-1)\left((-X+1)^{k-2-i}-(X+1)^{k-2-i}\right)=0. \]
Using binomial theorem, we can write this as 
\[ 2\sum_{\substack{1\le j\le k-3\\ j: odd}}\left(\sum_{\substack{2\le i\le k-4\\ i: even}}\left(\binom{i}{j}-\binom{i}{k-2-j}\right) a_{k-2-i} \right) X^j=0. \]
Therefore, the space $W_k^{+,0}$ is the set of polynomials 
\[P(X)=\sum_{\substack{2\le i\le k-4\\ i: even}} a_i X^i\]  whose (rational)
coefficients satisfy a set of linear relations
\begin{equation}\label{lineq} \sum_{\substack{2\le i\le k-4\\ i: even}}\left(\binom{i}{j}-\binom{i}{k-2-j}\right) a_{k-2-i}=0\quad (j=1,3,\ldots,k-3). \end{equation}
Clearly the equations for $j$ and $k-2-j$ have just opposite sign, and so we have to look only at the equations for $j\le k/2-1$.  
But then for $i<j$ the coefficient of $a_{k-2-i}$ is zero, and the coefficient matrix is upper triangular with non-zero diagonals.
We thus see that the rank of this matrix is $[(k+2)/4]-1$, and the dimension of $W_k^{+,0}$ is $k/2-2-\left([(k+2)/4]-1\right)=[k/4]-1$, 
as desired.
\end{proof}

\noindent{\bf Remark.}
We have not succeeded to characterize the (codimension 2) image of $S_k(2)$ by $r^+$ in $W_k^{+}\otimes\C$.
It is expected that such a characterization as in \cite{KoZa} should exist.\\

Amasingly enough, the coefficient matrix in \eqref{lineq} appears exactly when we look at the imaginary part of the double Eisensten series of level 2, which we are
going to explain.  We look at the imaginary part of $G_{r,k-r}^{\odd\odd}(\tau)$ (as $q$-series) for $r$ even. Let $\pi : \C[[q]] \longrightarrow \sqrt{-1}\R[[q]]$ be the
natural projection to imaginary part (note that by imaginary part we mean the term in $\sqrt{-1}\R[[q]]$, not the coefficient of $\sqrt{-1}$).  As Theorem~2
and Definition~1 (for $r=2$) shows, imaginary parts come from the terms with $\cz^\even(p)g_h^\odd(\tau)$ for odd $p$, and we have in matrix form
\begin{equation}\label{qmat}
\pi \left( \begin{array}{c} G_{2,k-2}^{\odd\odd}(q) \\ G_{4,k-4}^{\odd\odd}(q) \\ \vdots \\ G_{k-2,2}^{\odd\odd}(q) \end{array} \right) = Q_k\left( \begin{array}{c} \cz^{\even} (k-3) g_{3}^{\odd} (q) \\ \cz^{\even} (k-5) g_{5}^{\odd} (q) \\ \vdots \\ \cz^{\even} (3) g_{k-3}^{\odd} (q) \end{array} \right),
\end{equation}
where $Q_k$ is the $(k/2-1)\times (k/2-2)$ matrix given by
\[ Q_k=\left(\binom{2j}{2i-1} -  \binom{2j}{k-2i-1}\right)_{\substack{1\le i\le k/2-1\\ 1\le j\le k/2-2}}. \] 
This is exactly the coefficient matrix of \eqref{lineq}! 

Let $\DE$ be the $\Q$-vector space generated by $G_{r,k-r}^{\odd\odd} \ (r=2,4,\ldots,k-2)$. 

\begin{thm5}  Let $k\ge4$ be a positive even integer.

1) \[ \dim \DE = \frac{k}{2}-1,\]
so that the series $G_{r,k-r}^{\odd\odd}(\tau)$ $(r$ even$)$ are linearly independent over $\Q$.

2)  The space $\DE$ contains $\Q\cdot (2\pi i)^{-k}G_{k}^{(i\infty)}(\tau)\oplus S_k^{\Q} (2)$, 
where $S_k^{\Q} (2)$ is the space of cusp forms
on $\Gamma_0(2)$ having rational Fourier coefficients.
\end{thm5}
\begin{proof} We first prove 2). We know from Theorem~3 and Theorem~1 that the space $\DE$
contains $G_k^\odd(q)=(2\pi i)^{-k}G_{k}^{(i\infty)}(\tau)$, 
as well as $G_r^\odd(q)G_s^\even(q)$ and $G_r^\odd(q)G_s^\odd(q)$ ($r+s=k$).
Because of the relation 
\[ (2\pi i)^{-k}G_{r}^{(0)}(\tau) G_{s}^{(i\infty)}(\tau)=(2^r-1) G_r^\odd(q)G_s^\even(q) -G_r^\odd(q)G_s^\odd(q) 
\quad (q=e^{2\pi i \tau})\]
and the fact shown by Imamo\={g}lu and Kohnen in \cite{ki} that these cusp forms $G_{r}^{(0)}(\tau) G_{s}^{(i\infty)}(\tau)$ generate the space $S_k(2)$, we obtain the assertion 2).

For 1), first we note by definition the inequality
\[  \dim \DE \leq \frac{k}{2}-1.\]
Since elements in $\Q\cdot  (2\pi i)^{-k}G_{k}^{(i\infty)}(\tau)\oplus S_k^{\Q} (2)$ has no imaginary parts, they sit in the 
kernel of the projection $\pi$ from $\DE$ to $\sqrt{-1}\R[[q]]$, thus
\[ \dim\ker\pi\ge1+\dim S_k(2)=\left[\frac{k}{4}\right] . \]
As for the dimension of the image of $\pi$, we see that it is equal to the rank of the matrix $Q_k$
because the series $g_{3}^{\odd} (q), g_{5}^{\odd} (q),\ldots,g_{k-3}^{\odd} (q)$ are linearly
independent over $\C$.  This can be seen as follows. For a prime $p$, the coefficient of $q^p$ in
$g_r^\odd(q)$ is essentially $1+p^{r-1}$. Hence by picking distinct prime numbers
$p_3,p_5,\ldots,p_{k-3}$ and looking at the coefficients of $q^{p_3},q^{p_5},\ldots,q^{p_{k-3}}$
in $g_{3}^{\odd} (q), g_{5}^{\odd} (q),\ldots,g_{k-3}^{\odd} (q)$, we see the desired linear independence
because the coefficient matrix is essentially the Vandermond determinant.
We thus have 
\[ \dim \text{im } \pi =\text{rank} Q_k=\left[\frac{k+2}{4}\right] -1\]
and therefore 
\[  \dim \DE \geq \left[\frac{k}{4}\right]+\left[\frac{k+2}{4}\right] -1=\frac{k}{2}-1.\]
Therefore we conclude 
\[ \dim \DE =\frac{k}{2}-1\] and also
\[ \ker\pi=\Q\cdot  (2\pi i)^{-k}G_{k}^{(i\infty)}(\tau) \oplus S_k^{\Q} (2). \]

\end{proof}

\begin{cor}\label{5} 
 For an even integer $k>2$, we have
\[ \dim \langle \zeta^{\odd\odd} (2r,k-2r) \mid 1\leq r \leq k/2-1 \rangle_{\Q} \leq \frac{k}{2}-1-\dim S_k (2) .\]
\end{cor}

\begin{proof}  By taking the constant term of the $q$-series, we obtain the surjective map 
\[ \mu: \DE\longrightarrow \langle \zeta^{\odd\odd} (2r,k-2r) \mid 1\leq r \leq k/2-1 \rangle_{\Q}. \]
By the theorem, the kernel of $\mu$ contains the space $S_k^{\Q} (2)$ and hence we obtain the
corollary.
\end{proof}

\noindent{\bf Remark.}  The above corollary says that among the $k/2-1$ numbers
$\zeta^{\odd\odd}(even,\,even)$ there are at least  $\dim S_k (2)$ linear relations.  
It seems that the $k/2-1$ numbers $\zeta^{\odd\odd}(odd,\,odd)$ are linearly independent
over $\Q$, and the total space
\[ \langle \zeta^{\odd\odd} (r,k-r) \mid 2\leq r \leq k-1 \rangle_{\Q} \]
is spanned by $\zeta^{\odd\odd}(odd,\,odd)$ and $\zeta(k)$. (Recall the sum formula 
in Theorem 1, so that $\zeta(k)$ is contained in the above space.) 
The conjectural dimension of this space is thus $k/2$.
We also conjecture that the space of usual double zeta values of even weight $k$
is containend in the space spanned by $\zeta^{\odd\odd} (r,k-r)$ except $\zeta^{\odd\odd} (k-1,1)$:
\[ \langle \zeta(r,k-r) \mid 2\leq r \leq k-1 \rangle_{\Q}\subset \langle \zeta^{\odd\odd} (r,k-r) \mid 2\leq r \leq k-2 \rangle_{\Q}, \]
and that the space $\langle \zeta^{\odd\odd} (2r,k-2r) \mid 1\leq r \leq k/2-1 \rangle_{\Q} $
is contained in the usual double zeta space:
\[ \langle \zeta^{\odd\odd} (2r,k-2r) \mid 1\leq r \leq k/2-1 \rangle_{\Q} \subset 
\langle \zeta(r,k-r) \mid 2\leq r \leq k-1 \rangle_{\Q}. 
\]
When $k$ is odd, we can prove that every $\zeta^{\odd\odd} (r,k-r)$ except $\zeta^{\odd\odd} (k-1,1)$
is a linear combination of $\zeta(r,k-r)$:
\[ \langle \zeta^{\odd\odd} (r,k-r) \mid 2\leq r \leq k-2 \rangle_{\Q} \subset 
\langle \zeta(r,k-r) \mid 2\leq r \leq k-1 \rangle_{\Q}. 
\]
To prove this we use the identity of Y.~Komori, K.~Matsumoto, and H.~Tsumura
\begin{align*} &\left(1+(-1)^r\right)\zeta_2(r,s)+\left(1+(-1)^s\right)\zeta_2(s,r)\\
&\quad=\sum_{\substack{i=0\\even}}^{k-3}2^{-k+i+1}\left(\binom{k-i-1}{r-1}+\binom{k-i-1}{s-1}\right)\zeta(i)\zeta(k-i)-\zeta(k),
\end{align*}
valid when $r,s\ge2$ and $r+s=k:$ odd, where \[
\zeta_2(r,s)=\sum_{m,n\ge1}\frac1{(m+2n)^rm^s}=\zeta^{\odd\odd} (r,s)+\zeta^{\even\even} (r,s).\]
For $\zeta^{\odd\odd} (k-1,1)$, it seems we need $(\log2)\zeta(k-1)$ other than usual double zeta values,
but we have not proved this.

\subsection{Proof of Theorem~3}\label{proof}

As in \cite{gkz}, we prove Theorem~3 by dividing it into three parts: the constant term,  the imaginary part,
and the combinatorial part.  The double shuffle relation of the constant term is nothing but that of double zeta values,
namely Proposition~1 and its regularization.
As for  the imaginary part, the assertion is as follows.  

\begin{lem} \label{3_4} For each integer $k>2$, we define generating functions $\I^{\even\odd}_k (X,Y)$,
$\I^{\odd\even}_k (X,Y)$, $\I_k^{\odd\odd} (X,Y)$ by
\begin{align*}
\I^{\even\odd}_k (X,Y) &:= \sum_{r+s=k} I^{\even\odd}_{r,s} X^{r-1} Y^{s-1} , \ \I^{\odd\even}_k (X,Y) := \sum_{r+s=k} I^{\odd\even}_{r,s} X^{r-1} Y^{s-1} ,\\
\I_k^{\odd\odd} (X,Y) &:= \sum_{r+s=k} I^{\odd\odd}_{r,s} X^{r-1} Y^{s-1}.
\end{align*}
Then we have
\begin{align*}
\sum_{\begin{subarray}{c} p+h=k \\ p:odd \end{subarray} }  \left( X^{h-1} Y^{p-1}+X^{p-1} Y^{h-1} \right)  \cz^{\odd} (p) \overline{g}_h^{\odd} (q)  &= \I^{\odd\odd}_k (X,Y) +\I^{\odd\odd}_k (Y,X)\\
&= \I^{\even\odd}_k (X+Y,X)+\I^{\even\odd}_k (X+Y,Y) ,
\end{align*}
\begin{align*}
\sum_{\begin{subarray}{c} p+h=k \\ p:odd \end{subarray} } \left( X^{h-1} Y^{p-1} \cz^{\even} (p) g_h^{\odd} (q) +X^{p-1} Y^{h-1}  \cz^{\odd} (p) \overline{g}_h^{\even} (q) \right) & =\I^{\odd\even}_k (X,Y) +\I^{\even\odd}_k (Y,X)\\
&= \I^{\odd\even}_k (X+Y,X)+\I^{\odd\odd}_k (X+Y,Y) .
\end{align*}
\end{lem}

\begin{proof}
By definition, each generating function can be given as
\begin{align*}
\I^{\even\odd}_k(X,Y)= &\sum_{\begin{subarray}{c} p+h=k \\ p:odd \end{subarray}}  \left( X^{h-1}Y^{p-1} - X^{h-1} (X-Y)^{p-1} \right) \cz^{\odd} (p) g_h^{\even} (q) \\
 &+ \sum_{\begin{subarray}{c} p+h=k \\ p:odd \end{subarray}}   Y^{h-1} (X-Y)^{p-1}  \cz^{\odd} (p) g_h^{\odd} (q) ,\\
  \I^{\odd\even}_k(X,Y)=&  \sum_{\begin{subarray}{c} p+h=k \\ p:odd \end{subarray}} X^{h-1}Y^{p-1} \cz^{\even} (p) g_h^{\odd} (q) - \sum_{\begin{subarray}{c} p+h=k \\  p:odd \end{subarray}}  X^{h-1} (X-Y)^{p-1} \cz^{\odd} (p) g_h^{\odd} (q) \\
 &+  \sum_{\begin{subarray}{c} p+h=k \\ p:odd \end{subarray}} Y^{h-1}(Y-X)^{p-1} \cz^{\odd} (p) g_h^{\even} (q), \\
 \I^{\odd\odd}_k(X,Y)=&  \sum_{\begin{subarray}{c} p+h=k \\ p:odd \end{subarray}} \left( Y^{h-1}(Y-X)^{p-1} - X^{h-1} (Y-X)^{p-1} \right) \cz^{\even} (p) g_h^{\odd} (q) \\
\label{e_3_10} &+ \sum_{\begin{subarray}{c} p+h=k \\ p:odd \end{subarray}} X^{h-1}Y^{p-1} \cz^{\odd} (p) g_h^{\odd} (q) .
\end{align*}
The lemma follows from these by a simple calculation using binomial theorem and we omit the details.

\end{proof}

For the computation of the combinatorial part, we prepare the generating functions as follows.
\begin{align*}
\beta (X)&:= \sum_{p>0} \beta_p X^{p-1} =\frac{1}{2} \left( \frac{1}{X} - \frac{1}{e^X-1} \right), \\
\beta^{\even} (X) &:= \sum_{p>0} \beta_p^{\even} X^{p-1} =\frac{1}{4} \left( \frac{2}{X} - \frac{1}{e^{\frac{X}{2}}-1} \right), \\
\beta^{\odd} (X) &:= \sum_{p>0} \beta_p^{\odd} X^{p-1} = \frac{1}{4} \frac{1}{e^{\frac{X}{2}}+1},
\end{align*}
\begin{align*}
g^{\even}(X) &:= \sum_{p>0} g_p^{\even} (\tau) X^{p-1} = -\frac{1}{2} \sum_{u>0} e^{\frac{-uX}{2}}\cdot \frac{q^u}{1-q^u} ,\\
g^{\odd}(X) &:= \sum_{p>0} g_p^{\odd} (\tau) X^{p-1} = -\frac{1}{2} \sum_{u>0} (-1)^u e^{\frac{-uX}{2}} \cdot \frac{q^u}{1-q^u} ,\\
\overline{g}^{\even}(X) &:= \sum_{p>0} \overline{g}_p^{\even} (\tau) X^{p-1} = \frac{1}{2X}  \left( \sum_{u>0}  e^{\frac{-uX}{2}}\cdot \frac{q^u}{(1-q^u)^2} -4 g_2^{\even} (\tau) \right) ,\\
\overline{g}^{\odd}(X) &:= \sum_{p>0} \overline{g}_p^{\odd} (\tau) X^{p-1} = \frac{1}{2X} \left( \sum_{u>0} (-1)^u e^{\frac{-uX}{2}}\cdot \frac{q^u}{(1-q^u)^2} -4g_2^{\odd} (\tau) \right) ,
\end{align*}
\begin{align*}
g^{\even\odd} (X,Y) &:= \sum_{r,s\geq1} g_{r,s}^{\even\odd} (\tau) X^{r-1} Y^{s-1} = \frac{1}{4} \sum_{u,v>0} (-1)^v e^{-\frac{uX+vY}{2} }\cdot \frac{q^u}{1-q^u}\cdot \frac{q^{u+v}}{1-q^{u+v}}, \\
g^{\odd\even} (X,Y) &:= \sum_{r,s\geq1} g_{r,s}^{\odd\even} (\tau) X^{r-1} Y^{s-1} = \frac{1}{4} \sum_{u,v>0} (-1)^u e^{-\frac{uX+vY}{2} }\cdot \frac{q^u}{1-q^u}\cdot \frac{q^{u+v}}{1-q^{u+v}}, \\
g^{\odd\odd} (X,Y) &:= \sum_{r,s\geq1} g_{r,s}^{\odd\odd} (\tau) X^{r-1} Y^{s-1} = \frac{1}{4} \sum_{u,v>0} (-1)^{u+v} e^{-\frac{uX+vY}{2} }\cdot \frac{q^u}{1-q^u}\cdot \frac{q^{u+v}}{1-q^{u+v}}, 
\end{align*}
\begin{align*}
\beta^{\even \odd} (X,Y) &:= \sum_{r,s\geq 1} \beta_{r,s}^{\even\odd} (\tau) X^{r-1} Y^{s-1} =\beta^{\odd} (Y) g^{\even} (X) - \beta^{\odd} (X-Y) (g^{\even} (X)-g^{\odd} (Y) ), \\ 
\beta^{\odd \even} (X,Y) &:= \sum_{r,s\geq 1} \beta_{r,s}^{\odd\even} (\tau) X^{r-1} Y^{s-1} =\beta^{\even} (Y) g^{\odd} (X) - \beta^{\odd} (X-Y) (g^{\odd} (X)-g^{\even} (Y) ), \\ 
\beta^{\odd \odd} (X,Y) &:= \sum_{r,s\geq 1} \beta_{r,s}^{\odd\odd} (\tau) X^{r-1} Y^{s-1} =\beta^{\odd} (Y) g^{\odd} (X) - \beta^{\even} (X-Y) (g^{\odd} (X)-g^{\odd} (Y) ), 
\end{align*}
\begin{align*} 
\varepsilon^{\even\odd} (X,Y) &:=\sum_{r,s\geq1} \varepsilon_{r,s}^{\even\odd} (\tau) X^{r-1} Y^{s-1} \\
&= X\overline{g}^{\odd} (Y) -Y\overline{g}^{\odd} (Y) -\overline{g}_0^{\odd} (\tau) +X\overline{g}^{\even} (X) + \overline{g}_0^{\even} (\tau) +g^{\even} (X) +\alpha_1, \\
\varepsilon^{\odd\even} (X,Y) &:=\sum_{r,s\geq1} \varepsilon_{r,s}^{\odd\even} (\tau) X^{r-1} Y^{s-1} \\
&= X\overline{g}^{\even} (Y) -Y\overline{g}^{\even} (Y) -\overline{g}_0^{\even} (\tau) +X\overline{g}^{\odd} (X) + \overline{g}_0^{\odd} (\tau) +g^{\odd} (X) +\alpha_2, \\
\varepsilon^{\odd\odd} (X,Y) &:=\sum_{r,s\geq1} \varepsilon_{r,s}^{\odd\odd} (\tau) X^{r-1} Y^{s-1} = X\overline{g}^{\odd} (Y) -Y\overline{g}^{\odd} (Y) +X\overline{g}^{\odd} (X) +g^{\odd} (X) +\alpha_3,
\end{align*}
where $\alpha_1,\alpha_2$ and $\alpha_3$ are defined as in \eqref{e_3_8}. Let
\begin{align*} \CC^{\even\odd} (X,Y) &:=\sum_{r,s\geq1} C_{r,s}^{\even\odd} X^{r-1} Y^{s-1},\\
\CC^{\odd\even} (X,Y)& :=\sum_{r,s\geq1} C_{r,s}^{\odd\even} X^{r-1} Y^{s-1},\\
\CC^{\odd\odd} (X,Y) &:=\sum_{r,s\geq1} C_{r,s}^{\odd\odd} X^{r-1} Y^{s-1} .
\end{align*}
Then by definition we have
\begin{align*} \CC^{\even\odd}  (X,Y) &= g^{\even \odd} (X,Y) + \beta^{\even\odd} (X,Y) +\frac{1}{4} \varepsilon^{\even\odd} (X,Y) ,\\
\CC^{\odd\even} (X,Y) &= g^{\odd \even} (X,Y) + \beta^{\odd\even} (X,Y) +\frac{1}{4} \varepsilon^{\odd\even} (X,Y), \\
\CC^{\odd\odd}  (X,Y) &= g^{\odd \odd} (X,Y) + \beta^{\odd\odd} (X,Y) +\frac{1}{4} \varepsilon^{\odd\odd} (X,Y) .
\end{align*}
The double shuffle relations of the combinatorial double Eisenstein series are stated as
\begin{lem}\label{3_5} Put
\begin{align*}
\QQ^{\odd\even} (X,Y) &:= g^{\odd} (X) g^{\even} (Y) +\beta^{\odd} (X) g^{\even} (Y) +\beta^{\even} (Y) g^{\odd} (X) +\frac{1}{4} (X \overline{g}^{\even} (Y) + Y \overline{g}^{\odd} (X)), \\
\QQ^{\odd\odd} (X,Y) &:= g^{\odd} (X) g^{\odd} (Y) +\beta^{\odd} (X) g^{\odd} (Y) +\beta^{\odd} (Y) g^{\odd} (X) +\frac{1}{4} (X \overline{g}^{\odd} (Y) + Y \overline{g}^{\odd} (X)) ,\\
\CC^\odd (X) &:= g^{\odd} (X) -\frac{\alpha_3}{2} \cdot X.
\end{align*}
Then we have
\begin{align*}
\QQ^{\odd\even} (X,Y) &= \CC^{\odd\even} (X,Y) + \CC^{\even\odd} (Y,X)  = \CC^{\odd\even} (X+Y,Y)+\CC^{\odd\odd}  (X+Y,X),\\
\QQ^{\odd\odd}(X,Y) &= \CC^{\odd\odd}  (X,Y) + \CC^{\odd\odd}  (Y,X) +\frac{\CC^\odd (X) -\CC^\odd (Y) }{X-Y} = \CC^{\even\odd}  (X+Y,X)+\CC^{\even\odd}  (X+Y,Y).
\end{align*}
\end{lem}
\begin{proof}
Computations are parallel to those in \cite{gkz}, though tedious, and we omit the details.
\end{proof}
The two lemmas and Proposition 1 complete the proof of Theorem 3. 

\appendix
\section{The double Eisenstein series and the period polynomials in the case of $\SL$}\label{appendix}

In this appendix we briefly recall the relation described in \cite{kaneko} between the double Eisenstein series and
modular forms for $\SL$.  

The double Eisenstein series for $\SL$ was first defined and studied in \cite{gkz}:
\[ G_{r,s} (\tau) := (2\pi i)^{-r-s}\!\!\!\sum_{\begin{subarray}{c} \lambda > \mu>0 \\ \lambda, \mu \in \Z\cdot\tau + \Z \end{subarray}} 
\frac{1}{\lambda^r \mu^s} =(2\pi i)^{-r-s}\!\!\!\!\!\sum_{\begin{subarray}{c} m\tau+n > m'\tau+n' >0 \\ m, n, m', n' \in \Z \end{subarray}} 
\frac{1}{(m\tau+n)^r (m'\tau+n')^s}.\]
Its Fourier series is given there as 
\[ G_{r,s}(\tau) =  \cz(r,s) + g_{r,s} (q) 
+\sum_{ \begin{subarray}{c} p+h=k \\ p>1 \end{subarray} }  \left((-1)^s \binom{p-1}{s-1} +(-1)^{p+r} \binom{p-1}{r-1}+\delta_{p,s}\right) \cz(p) g_h(q) , \]
where $\cz(r,s) =(2\pi i)^{-r-s} \zeta(r,s), \cz(p)=(2\pi i)^{-p}\zeta(p)$, and
\begin{align*}
g_{r,s}(q) &= \frac{(-1)^{r+s}}{(r-1)!(s-1)!} \sum_{\begin{subarray}{c} m>n>0 \\ u,v>0 \end{subarray} } u^{r-1} v^{s-1} q^{um+vn},\\
g_h(q)&= \frac{(-1)^h}{(h-1)!}\sum_{u,m>0} u^{h-1} q^{um}.
\end{align*} 
By extending the definition in the case of  
non-absolute convergence using $q$-series, we showed that the double Eisenstein series satisfy the double shuffle
relations (in the form described in \cite{gkz}), that the space of double Eisenstein series contains the space of modular forms on $\SL$, and made
a connection to the period polynomial by looking at the imaginary parts of the $q$-expansions of $G_{r,s}(\tau)$.
Specifically, the imaginary parts are given, like \eqref{qmat}, by
\begin{equation*}
\pi \left( \begin{array}{c} G_{2,k-2}(\tau) \\ G_{3,k-3}(\tau) \\ \vdots \\ G_{k-2,2}(\tau) \end{array} \right) = Q_k^{(1)}\left( \begin{array}{c} \cz^ (k-3) g_{3}(q) \\ \cz (k-5) g_{5} (q) \\ \vdots \\ \cz(3) g_{k-3} (q) \end{array} \right),
\end{equation*}
where $Q_k^{(1)}$ is the $(k-3)\times (k/2-2)$ matrix given by
\[ Q_k^{(1)}=\left((-1)^i\binom{2j}{i} - (-1)^i \binom{2j}{k-2-i}+\delta_{k-2-i,2j}\right)_{\substack{1\le i\le k-3\\ 1\le j\le k/2-2}}. \] 
Rather surprisingly, this contains exactly $Q_k$ as a minor.
For example, 
\[ Q_{12}^{(1)} = \left( \begin{matrix} -2 & -4 & -6 & -8 \\ 1&6&15&28\\0&-4&-20&-48\\0&1&15&42\\0&0&0&0\\0&0&-14&-42\\0&4&20&48\\0&-6&-15&-27\\2&4&6&8 \end{matrix} \right),\ \ 
Q_{12} = \left( \begin{array}{cccc} -2&-4&-6&-8\\0&-4&-20&-48\\0&0&0&0\\ 0&4&20&48\\2&4&6&8  \end{array} \right). \]
Precisely, the $i$-th row of $Q_k$ is the $2i-1$-st row of $Q_k^{(1)}$. 

The right kernel of $Q_{k}^{(1)}$ corresponds to the even period polynomials (without constant term) of weight $k$
for $\SL$, an example being $\tr (1,-3,3,-1)$ in the right kernel of $Q_{12}^{(1)}$ and the corresponding period polynomial
$X^8-3X^6+3X^4-X^2$ of weight $12$.  As in the corollary of Theorem~5, by looking at the constant
term of the double Eisenstein series and by using the connection to the period polynomial just mentioned, we 
obtain the upper bound of the dimension of the space of double zeta values:
\[ \dim \langle \zeta(r,k-r) \mid 2\leq r \leq k-1 \rangle_{\Q} \leq \frac{k}{2}-1-\dim S_k (1) .\]

Also,  elements in the left kernel of $Q_{k}^{(1)}$ produce expressions of modular forms 
in terms of double Eisenstein series. By comparing the Fourier coefficients, we obtain 
certain formulas for Fourier coefficients of modular forms. Let us look at some examples in weight $12$.

As the simplest example, take $(0,0,0,0,1,0,0,0,0)$ in 
the left kernel of $Q_{12}^{(1)}$. This corresponds to the relation 
\[   
2^7\cdot3\cdot5^2\cdot691\,G_{6,6}(\tau)=2^9\cdot3^2\cdot5^2\,{\widetilde G}_{12}(\tau)-\Delta(\tau),
\]
where $\Delta (\tau) = q\Pi_{n>0} (1-q^n)^{24} = \sum_{n>0} \tau(n) q^n$ is the famous cusp form of weight $12$. 
Comparing the coefficients of both sides, we obtain
\[ \tau(n)= \frac{2}{693} \sigma_{11} (n) + \frac{691}{2^2\cdot 3^2 \cdot 7} \sigma_5 (n) - \frac{691}{2^2\cdot 3^2} \sigma_3 (n) + \frac{5\cdot 691}{2\cdot 3^2\cdot 11} \sigma_1(n) -\frac{2\cdot 691} {3} \rho_{5,5} (n), \]
where 
\[ \rho_{k,l} (n) := \sum_{\begin{subarray}{c} a+b=n\\ a,b>0 \end{subarray} } \sum_{\begin{subarray}{c} u|a,v|b \\ \frac{a}{u}>\frac{b}{v} \end{subarray} } u^k v^l . \]
Incidentally, the Ramanujan congruence 
\[ \tau(n) \equiv \sigma_{11} (n) \pmod{691} \]
is clearly seen  from this.
Secondly take $(0,0,7,28,0,20,0,0,0)$, which gives the relation  
\[2^7\cdot3^2\cdot5\cdot7\cdot691G_{4,8}(\tau)+2^9\cdot3^2\cdot5\cdot7\cdot691G_{5,7}(\tau)+
2^9\cdot3^2\cdot5^2\cdot691G_{7,5}(\tau)=2^5\cdot3^3\cdot5\cdot11\cdot149{\widetilde G}_{12}(\tau)
-\Delta(\tau) \]
and the formula
\begin{align*} \tau(n)&=\frac{149}{840}\sigma_{11}(n)
-\frac{691}{180}\sigma_7(n)
-\frac{11747}{126}\sigma_5(n)
 +\frac{173441}{360}\sigma_3(n)-\frac{3455}{9}
\sigma_1(n)\\
&\quad  -\frac{2764}3\rho_{3,7}(n)-\frac{19348}3\rho_{4,6}(n)
-\frac{13820}3\rho_{6,4}(n).
\end{align*}
As the third example, take $(0,0,0,168,0,150,0,28,0)$. This corresponds to the relation (8) in \cite{gkz}, and
gives
\[2^9\cdot3\cdot5\cdot7\cdot691G_{5,7}(\tau)+2^7\cdot3\cdot5^3\cdot691G_{7,5}(\tau)+
2^8\cdot5\cdot7\cdot691G_{9,3}(\tau)=2^6\cdot3^3\cdot5\cdot191{\widetilde G}_{12}(\tau)
-\Delta(\tau) \]
and 
\begin{align*} \tau(n)&=\frac{5197}{124740}\sigma_{11}(n)
+\frac{691}{270}\sigma_7(n)
-\frac{129217}{2268}\sigma_5(n)
 +\frac{57353}{270}\sigma_3(n) -\frac{3455}{22}
\sigma_1(n)\\
 &\quad-\frac{19348}9\rho_{4,6}(n)-\frac{17275}9\rho_{6,4}(n)
-\frac{691}9\rho_{8,2}(n).
\end{align*}

We may take yet other vectors in the left kernel of $Q_{12}^{(1)}$ (the dimension is $6$) and
may deduce similar kind of formulas for $\tau(n)$.\\

\noindent{\bf Remark.}  Interestingly enough, the matrix $Q_{k}^{(1)}$ appears when we write 
``motivic'' double zeta values in terms of certain basis elements $f_3,\,f_5,\ldots$ using coproduct structure
described in F.~Brown's recent important papers \cite{B1,B2}. One of the present authors has found 
the same relation between triple Eisenstein series and motivic triple zeta values.  
Or a variant (minor matrix) of $Q_{k}^{(1)}$
appears in the work of S.~Baumard and L.~Schneps \cite{sl} on a relation of double zeta values and 
period polynomials.  More precisely, the matrix $Q_{k}^{(1)}\setminus Q_k$ (we view a matrix as a union of
row vectors and take the difference of sets) is the matrix $A$ in \cite{sl}. The vector $(0,0,0,168,0,150,0,28,0)$ above
is essentially the unique vector $(0,168,150,28)$ in the kernel of ${}^t\!A$.



\end{document}